\documentclass{amsart}

\usepackage{fullpage}
\usepackage{amsmath}

\usepackage{amsfonts}
\usepackage{amssymb}

\input xy
\xyoption{matrix}\xyoption{arrow}\xyoption{curve}

\DeclareMathOperator{\oGr}{Gr}
\DeclareMathOperator{\ogr}{gr}
\mathchardef\mhyphen="2D
\DeclareMathOperator{\orMod}{Mod}
\DeclareMathOperator{\ormod}{mod}

\DeclareMathOperator{\lmod}{\mhyphen mod}

\newcommand{\Hom}{\mathrm{Hom}}

\newcommand{\Ext}{\mathrm{Ext}}

\newcommand{\Endo}{\mathrm{End}}

\newcommand{\Tr}{\mathrm{Tr}}

\newcommand{\coker}{\mathrm{coker}}

\newcommand{\add}{\mathrm{add}}
\newcommand{\Add}{\mathrm{Add}}
\newcommand{\ind}{\mathrm{ind}}
\newcommand{\rad}{\mathrm{rad}}

\newcommand{\pp}{\mathrm{pp}}

\newcommand{\Gr}[1][]{\oGr^{#1 \!} \mhyphen}

\newcommand{\gr}[1][]{\ogr^{#1 \!} \mhyphen}
\newcommand{\lgr}[1][]{\mhyphen \ogr^{#1 \!}}
\newcommand{\rmod}{\ormod \mhyphen}
\newcommand{\rMod}{\orMod \mhyphen}

\newcommand{\gl}{\mbox{gr.l.}}

\newcommand{\ses}[5]{\ensuremath{0 \rightarrow #1 \stackrel{#4}{\longrightarrow} 
#2 \stackrel{#5}{\longrightarrow} #3 \rightarrow 0}}
\newcommand{\A}{\mathcal{A}}

\newcommand{\C}{\mathcal{C}}

\newcommand{\G}{\mathcal{G}}

\newtheorem{therm}{Theorem}[section]
\newtheorem{defin}[therm]{Definition}
\newtheorem{propos}[therm]{Proposition}
\newtheorem{lemma}[therm]{Lemma}
\newtheorem{coro}[therm]{Corollary}
\newtheorem{remark}[therm]{Remark}
\newtheorem{question}[therm]{Question}

\begin{document}

\title{Gradable modules over artinian rings}
\author{Alex Dugas}
\address{Department of Mathematics, University of the Pacific, 3601 Pacific Ave, Stockton, CA 95211, USA}
\email{adugas@pacific.edu}


\begin{abstract} Let $\Lambda$ be a $\mathbb{Z}$-graded artin algebra.  Two classical results of Gordon and Green state that if $\Lambda$ has only finitely many indecomposable gradable modules, up to isomorphism, then $\Lambda$ has finite representation type; and if $\Lambda$ has finite representation type then every $\Lambda$-module is gradable.  We generalize these results to $\mathbb{Z}$-graded right artinian rings $R$.  The key tool is a characterization of gradable modules: a f.g. right $R$-module is gradable if and only if its ``pull-up'' is pure-projective.  Using this we show that if there is a bound on the graded-lengths of f.g. indecomposable graded $R$-modules, then every f.g. $R$-module is gradable.  As another consequence, we see that if a graded artin algebra has an ungradable module, then it has a Pr\"ufer module which is not of finite type, and hence it has a generic module by work of Ringel.  
\end{abstract}

\maketitle

\section{Introduction}

Let $R = \oplus_{n} R_n$ be a $\mathbb{Z}$-graded ring.  We are interested in comparing the ordinary representation theory of $R$, i.e., the categories $\rmod R$ (resp.  $\rMod R$) of finitely generated (resp. all) right $R$-modules, to the { \it graded} representation theory of $R$ expressed by the categories $\gr R$ (resp. $\Gr R$) of finitely generated (resp. all) $\mathbb{Z}$-graded right $R$-modules.  We have a natural forgetful functor $q : \Gr R \rightarrow \rMod R$ that forgets the grading of a graded module.  We call an $R$-module $M$ {\it gradable} if it belongs to the strict image of this functor.  In this paper we consider the following finiteness properties on the category of finitely generated right $R$-modules when $R$ is right artinian.

\begin{itemize}
\item[(F1)] $R$ has finite representation type; i.e. there exist only finitely many f.g. indecomposable $R$-modules up to isomorphism.
\item[(F2)] There exist only finitely many f.g. indecomposable gradable $R$-modules up to isomorphism.
\item[(F3)] There exists a bound on the graded lengths of the f.g. indecomposable graded $R$-modules.
\item[(F4)] Every f.g. $R$-module is gradable.
\end{itemize}

The implication $(F1) \Rightarrow (F2)$ is trivial, while $(F2) \Rightarrow (F3)$ is a consequence of a theorem of Camillo and Fuller \cite{CF} (see Lemma~\ref{lemma:F2F3}).  Here, the {\it graded length} of a nonzero graded $R$-module $M = \oplus_n M_n$ is defined as $$\gl M = \sup \{n \mid M_n \neq 0\} - \inf \{n \mid M_n \neq 0\} + 1 \in \mathbb{N}\cup \{\infty\}.$$

When $R$ is an artin algebra, Gordon and Green have shown that $(F2) \Rightarrow (F4)$ and hence that $(F2) \Rightarrow (F1)$ as well \cite{GoGr2}.  They speculate also that $(F3) \Rightarrow (F4)$  in this case.   In \cite{CF}, Camillo and Fuller generalize some of Gordon's and Green's results to graded right artinian rings, and they raise the questions of whether $(F2) \Rightarrow (F1)$ or $(F1) \Rightarrow (F4)$ might remain true in this setting.  Our main result provides affirmative answers to both questions by showing that $(F3) \Rightarrow (F4)$.

\begin{therm}\label{thm:FiniteG} Let $R$ be a graded right artinian ring.  If $$G_R = \sup \{ \gl(M)\ |\ M \in \gr R\ \mbox{indecomposable}\} < \infty,$$ then every f.g. right $R$-module is gradable.   Consequently, $R$ has finite representation type if and only if there are only finitely many indecomposable gradable $R$-modules, up to isomorphism; and in this case, every f.g. $R$-module is gradable.
\end{therm}

Our proof relies on a characterization of gradable modules in terms of their {\it pull-ups}, which are infinitely generated graded modules defined using the right adjoint $p$ of the forgetful functor $q$.  We review the definitions and basic properties of these modules in Section 2 for arbitrary graded rings.   In Section 3, we specialize to graded right artinian rings and prove our main characterization of gradable modules in this context.

\begin{therm}[Theorem~\ref{thm:GradablePP}, Lemma~\ref{lemma:pullupGraded}] Let $R$ be a right artinian graded ring.  A fintely generated right $R$-module $N$ is gradable if and only if its pull-up $pN$ is pure-projective if and only if $qpN \cong N^{(\omega)}$.
\end{therm}
 
 Next, in Section 4, we investigate some consequences of $G_R < \infty$ in order to prove Theorem~\ref{thm:FiniteG}.  In Section 5, we briefly consider pure semisimple rings.  Motivated by the pure-semisimplicity conjecture, we show that (F3) and (F4) are left-right symmetric.  Afterwards, in Section 6 we investigate connections to the second Brauer-Thrall conjecture for artin algebras using Pr\"ufer modules as indicated by Ringel \cite{Ringel2}.  Namely, we show that if a graded artin algebra has an ungradable module, then the pull-up of this module can be used to construct a Pr\"ufer module which is not of finite type.  This construction was one of the motivating problems for this work. Finally, we conclude with several open questions, some of which are inspired by results of Gordon and Green on the AR-quivers of graded artin algebras.

 \section{Background on graded modules}
 
  Let $R = \oplus_{n \in \mathbb{Z}} R_n$ be a graded ring.  Observe that $R_0$ is a unital subring of $R$, often called the initial subring of $R$.  A right $R$-module $M_R$ is $\mathbb{Z}$-graded if there exists a decomposition $M = \oplus_{n \in \mathbb{Z}} M_n$ (of abelian groups) such that $M_i R_j \subseteq M_{i+j}$ for all $i$ and $j$.  (As we only consider $\mathbb{Z}$-graded modules here, we will omit the $\mathbb{Z}$ and just write `graded modules'.)  An $R$-module map $f : M \rightarrow N$ between two graded $R$-modules is {\it homogeneous} (or {\it degree zero}) if $f(M_n) \subseteq N_n$ for all $n \in \mathbb{Z}$.  We write $\rMod R$ for the category of all right $R$-modules, and $\Gr R$ for the category of all graded right $R$-modules and degree zero morphisms.  The category $\Gr R$ is endowed with an automorphism $S$ defined by shifting the grading of a graded module $M$: namely, $SM = M = \oplus M_n$ but with $(SM)_n = M_{n+1}$ for all $n$, and $S(f) = f$ for all morphisms $f$.  We may also write $M[d]$ for $S^dM$.

  A graded module $M_R$ is {\it bounded below} (resp. {\it bounded above}) if $M_n = 0$ for $n \ll 0$ (resp. for $n \gg 0$).  Observe that if the grading on $R$ is bounded above (resp. below), then any f.g. graded $M_R$ is bounded above (resp. below).  If $M_R$ is both bounded above and below, we say $M$ is {\it bounded} (or {\it finitely graded}) and we define the {\it graded length} of $M \neq 0$ as $$\gl M = \max \{n\ |\ M_n \neq 0\} - \min \{n\ |\ M_n \neq 0\}+1.$$
Each homogeneous piece $M_i$ of a graded module $M_R$ is a right $R_0$-module, and we say that $M_R$ is {\it locally finite} if each $M_i$ has finite length over $R_0$.  Additionally, we will say that $M_R$ is {\it finite} if it is locally finite and bounded.  Notice that this is equivalent to $M$ having finite length as an $R_0$-module.  Additionally, using interval notation, we say that a graded module $M$ is concentrated in degrees $[a,b]$ if $M = \oplus_{n=a}^b M_n$. 
   
  We have an exact functor $q: \Gr R \rightarrow \rMod R$ which simply forgets about the grading of a graded module $M$.  We say that an $R$-module $N$ is {\it gradable} if it is isomorphic to $qM$ for some graded $M_R$, and we write $\G_R$ (resp. $\G^f_R$) for the full subcategory of gradable (resp. finitely generated, gradable) $R$-modules.

 In the literature on Galois coverings \cite{BG}, $q$ is called a {\it  push-down} functor, and it has a right adjoint $p$, called a {\it pull-up} functor.  We now describe $p: \rMod R \rightarrow \Gr R$ in more detail.  From an arbitrary $R$-module $M$, we can define a graded $R$-module $pM$ with $(pM)_n = M$ for all $n$.  We write $$pM = \oplus_{n \in \mathbb{Z}} Me_n,$$
 where the $e_n$ are formal symbols that we use to keep track of the different summands of $pM$ (for any $M$).  The $R$-action on $pM$ is given by $$(me_n)r = (mr)e_{n+j}, \ \ \forall\ r \in R_j$$ and extended linearly.  For a morphism $f : M \rightarrow N$, we define $pf : pM \rightarrow pN$ by $pf(me_i) = f(m)e_i$ for all $i$.  Clearly $pf$ is homogeneous morphism of graded $R$-modules.  It is also easy to see that $p$ is an exact functor.  We note that $pM$ is always $S$-invariant, meaning $pM \cong pM[1]$ via a natural map $\sigma_M : pM \cong pM[1]$ that sends $me_i \mapsto me_{i+1}$ for each $m \in M$ and $i \in \mathbb{Z}$.
 
 Furthermore, we have natural transformations $$\delta : qp \rightarrow 1_{\rMod R}\ \mbox{and}\ \epsilon : 1_{\Gr R} \rightarrow pq,$$ which are given by $\delta_N(\sum_i n_i e_i) = \sum_i n_i$ for any $R$-module $N$ with $n_i \in N$, and $\epsilon_M (m)  = me_i$ for any graded module $M_R$ and $m \in M_i$.  It is not difficult to check that $\delta$ and $\epsilon$ provide the co-unit and unit, respectively, for an adjunction $q \dashv p$.
 
 \begin{lemma} \begin{enumerate}
 \item The pull-up functor $p : \rMod R \rightarrow \Gr R$ is right adjoint to the push-down functor $q : \Gr R \rightarrow \rMod R$, i.e., for all $M_R$ graded and all $N_R$, we have natural isomorphisms $$\Hom_R(qM,N) \cong \Hom_{\Gr R}(M, pN).$$ 
 \item We also have natural isomorphisms 
 $$\Hom_R(N,qM) \cong \Hom_{\Gr R}(pN,M)$$
 whenever $M$ has finite graded length.  Thus, in a sense, $p$ is close to being left adjoint to $q$, but is not in general.

\end{enumerate}
 \end{lemma} 
  
 \noindent
\begin{proof}   (1) Let $M =\oplus_n M_n$ be a graded $R$-module and consider a map $f: qM \rightarrow N$.  We define a homogeneous map $\eta f : M \rightarrow pN$ by $\eta f (m) = f(m)e_i$ for all $m \in M_i$, and extend $\eta(f)$ linearly to all of $M$.  If $g : M \rightarrow pN$ is a homogeneous map, the inverse of $\eta$ is given by $\eta^{-1}g = \delta_N \circ q(g) : qM \rightarrow qpN \rightarrow N$.
  
  (2) We start by defining a natural map $\zeta : \Hom_R(N,qM) \rightarrow \Hom_{\Gr R}(pN,M)$ for any $N_R$ and any graded $M_R$.  For a graded module $M$ let $\pi_i : M \rightarrow M_i$ denote the projection as a map of abelian groups.  If $f : N \rightarrow qM$, we define a homogeneous map $\zeta f : pN \rightarrow M$ by $\zeta f( n e_i) = \pi_i f(n)$ for each $i \in \mathbb{Z}$ and $n \in N$ and extending linearly.  To see that $\zeta f$ is $R$-linear, note that for $r \in R_j$ we have $$\zeta f( ne_i \cdot r) = \zeta f (nr e_{i+j}) = \pi_{i+j} f(nr) = \pi_{i+j} (f(n)r),$$
    and since $M$ is graded, $$\pi_{i+j}(f(n) r) = \pi_{i+j}(\pi_i(f(n)) r) = \pi_i(f(n)) r = \zeta f (ne_i) \cdot r.$$
 Now to define an inverse of $\zeta$, we must assume that $M$ has finite graded length.  In this case, for a homogeneous map $g: pN \rightarrow M$, we set $\zeta^{-1}g(n) = \sum_i g(n e_i)$, noting that the sum is finite because $M$ is nonzero in only finitely many degrees. 
 \end{proof}
  
  \begin{coro}\label{coro:counitSplits}
  \begin{enumerate}
  \item For any $N_R$, the co-unit $\delta_N : qp N \rightarrow N$ provides a right $\G_R$-approximation of $N$.  That is, any map $f : M \rightarrow N$ with $M$ gradable factors through $\delta_N$ via a map $M \rightarrow qp N$.
\item  If $N_R$ is gradable then the co-unit $\delta_N : qp N \rightarrow N$ splits.
\end{enumerate}
  \end{coro}

  \noindent
\begin{proof}  (1) Let $f : M \rightarrow N$ be a map with $M = qM'$ for a graded module $M'$.  Then $\eta(f) : M' \rightarrow pN$ is a homogeneous map and $f = \delta_N q (\eta(f))$ by standard properties of adjoints.   Thus $f$ factors through $\delta_N$.
  
  (2) If $N_R$ is gradable, then (1) applies to the identity map $1_N$.
  \end{proof}

  The above corollary shows that if $N$ is gradable, then it is isomorphic to a direct summand of $qp N$.  In fact, we can say even more in this case.
  
  \begin{lemma}\label{lemma:pullupGraded}  \begin{enumerate}
  \item For any graded module $M_R$, $pq M \cong \oplus_{i \in \mathbb{Z}} M[i]$.
  \item Suppose that $N_R$ is gradable, with $N \cong qM$ for a graded module $M$.  Then $pN \cong pq M \cong \oplus_{i \in \mathbb{Z}} M[i]$.
Consequently $qp N \cong N^{(\omega)}$.  
\end{enumerate}
  \end{lemma}
  
  \noindent
  \begin{proof}  (1) For each $i \in \mathbb{Z}$, define a homogeneous map $f_i : M[i] \rightarrow pq M$ by $f_i(m) = m e_{j-i}$ for all $m \in M_j = (M[i])_{j-i}$.  In fact, we have $f_i = \sigma_{qM}^{-i} f_0[i]$ for each $i \in \mathbb{Z}$, where $f_0$ coincides with the splitting of $\delta_{qM}$.   Altogether the maps $f_i$ for $i \in \mathbb{Z}$ induce a map $f : \oplus_{i \in \mathbb{Z}} M[i] \rightarrow pq M$, which is easily seen to be an isomorphism in each degree, and thus an isomorphism.
  
  (2) is an immediate consequence of (1).
   \end{proof}
    
  \begin{coro}\label{coro:isomorphicGradable} If $M$ and $M'$ are strongly indecomposable graded modules, then $qM \cong qM'$ if and only if $M' \cong M[i]$ for some $i \in \mathbb{Z}$.
  \end{coro}

\begin{proof} Assume $M$ and $M'$ are stongly indecomposable graded modules.  If $qM \cong qM'$, then $\oplus_{i \in \mathbb{Z}} M[i] \cong pqM \cong pqM' \cong \oplus_{i \in \mathbb{Z}} M'[i]$.  Since $M'$ is strongly indecomposable, it must be isomorphic to a direct summand of some $M[i]$.  As $M[i]$ is indecomposable, we have $M' \cong M[i]$ in $\Gr R$.  The converse is clear. 
\end{proof}
  
  We are interested in the converse of Corollary~\ref{coro:counitSplits}(2).  Namely, if $\delta_N$ splits does it follow that $N$ is gradable?  Notice that it is not automatic that a direct summand of a gradable module is also gradable.  In particular, if $N_R$ is a non-gradable projective module, then $\delta_N$ splits since it is onto, and we obtain a counterexample.   To be more concrete, such non-gradable projective modules exist over polynomial rings of the form $R=D[x,y]$, with the usual grading, when $D$ is a noncommutative division ring \cite{SPPM}.  However, if $R$ is right artinian, then the converse does hold as we will see in the next section.  This fact is an easy consequence of results of Camillo and Fuller, the first of which we state now since it does not require any additional assumptions on the ring $R$.

\begin{therm}[Corollary 2, Theorem 1 in \cite{CF}]\label{thm:CF} Let $M_R$ be a finitely graded $R$-module with a.c.c. and d.c.c. on homogeneous submodules.  Then $M$ is indecomposable in $\Gr R$ if and only if $qM$ is indecomposable in $\rMod R$.  Moreover, in this case $S=\Endo_R(qM)$ is a finitely graded local ring with $\rad S = \rad S_0 \oplus \bigoplus_{n\neq 0} S_n$.
\end{therm}

We close this section with an important result concerning finitely graded direct summands of pull-up modules.

\begin{propos}\label{prop:NgeneratesM}
Let $N_R$ be an $R$-module and $M_R$ a locally finite graded $R$-module of finite graded length.  If there exists a split epimorphism $g : pN \rightarrow M$ in $\Gr R$, then $qM$ belongs to $\add(N)$.
\end{propos} 

\begin{proof}  We may assume that $M$ is indecomposable in $\Gr R$.  Note that the assumptions on $M$ imply that $M$ has finite length over $R_0$, and hence the hypotheses of the above theorem are satisfied.  Since $M$ has finite graded length, we may assume that $M$ is concentrated in degrees $[0,d]$ for some $d \geq 0$.  The full endomorphism ring $S=\Endo_R(qM)$ is then graded with $S_j \cong \Hom_{\Gr R}(M,M[j])$ and $S_i = 0$ whenever $|i|>d$.  By the above theorem, it is local and $\rad S = \rad S_0 \oplus \bigoplus_{j\neq 0} S_j$. 

Let $i : M \to pN$ be a splitting for $g$ in $\Gr R$.  For each $k \in \mathbb{Z}$, we define a degree $k$-endomorphism of $M$ by
$$f_k = g[k]\circ \sigma_N^k \circ i,$$ where $\sigma_N : pN \to pN[1]$ is the isomorphsim sending $ne_j$ to $ne_{j+1}$ for all $n\in N$ and all $j \in \mathbb{Z}$.  Clearly, $f_0 = gi = 1_M$, and $f_k \in S_k \subseteq \rad S$ for all $k\neq 0$.  In particular, $f = \sum_{|k| \leq d} f_k$ is a unit in $S$.

Now let $m \in M_j$ be a homogeneous element of $M$ and let $i(m) = ne_j$ for some $n \in N$ and $0\leq j \leq d$.  Let $h = \zeta^{-1}(g): N \to qM$, so that $h(n) = \sum_{k=0}^d g(ne_k)$.  Observe that $g(ne_k) = g[k-j]\sigma_N^{k-j}(ne_j) = f_{k-j}(m)$ for each $k$.  Thus
$$h(n) = \sum_{k=0}^d g(ne_k) = \sum_{k=-j}^{d-j}f_k(m) = \sum_{|k|\leq d} f_k(m) = f(m).$$
Thus $m = f^{-1}h(n)= f^{-1}h\delta_N i(m)$, which shows that $f^{-1}h$ splits the map $\delta_N q(i) : qM \to N$.  Hence $qM$ is isomorphic to a direct summand of $N$. 
\end{proof}

\section{Graded modules over artinian rings}

From now on we assume $R = \oplus_n R_n$ is a right artinian graded ring (unless otherwise noted).  To be clear, throughout this article this will mean that $R$ is a right artinian ring which also happens to be graded (we never work under the weaker assumption that $R$ is graded and satisfies DCC on homogeneous right ideals).  In \cite{CF}, Camillo and Fuller show that a graded ring $R$ is right artinian if and only if its initial subring $R_0$ is right artinian and $R$ is finitely generated as a right $R_0$-module.  In particular, when $R$ is right artinian, we know that $R_n =0$ for $|n| \gg 0$.  We write $J = \rad R$ for the Jacobson radical of $R$, which is a homogeneous ideal by a well-known result of Bergman \cite{B}.   If $R$ is basic, $R/J$ is a direct product of division rings, which must be trivially graded by Proposition 4 of \cite{CF}.  Thus, if $R$ is basic or if $R$ is positively graded (i.e., $R_n = 0$ for all $n<0$), we have $J = \rad R_0 \oplus \bigoplus_{n \neq 0} R_n$.  In either case, it follows that all simple $R$-modules are gradable, and each graded simple module is concentrated in a single degree (see also Corollaries 4.4 and 4.5 in \cite{GoGr1}).  For $R$ right artinian any f.g. graded $R$-module is bounded and locally finite.  Camillo and Fuller \cite{CF} also show that a f.g. graded $R$-module is (strongly) indecomposable in $\Gr R$ if and only if it is (strongly) indecomposable in $\rMod R$ (Cor. 6), and that any f.g. $R$-module that is either semisimple, projective, injective or a direct summand of a f.g. gradable module is gradable (Prop. 7). 

For the reader's convenience, we include a simple argument that $(F2) \Rightarrow (F3)$ for right artinian graded rings.  Recall that $G_R$ is defined to be the supremum of the graded lengths of the finitely generated indcomposable graded $R$-modules.

\begin{lemma}\label{lemma:F2F3}
Let $R$ be a right artinian graded ring with only finitely many indecomposable gradable modules up to isomorphism.  Then $G_R < \infty$.
\end{lemma}

\begin{proof} Suppose $qM_1, \ldots, qM_n$ are all the f.g. indecomposable gradable $R$-modules up to isomorphism, where $M_1, \ldots, M_n$ are f.g. graded $R$-modules.  Clearly, each $M_i$ is indecomposable in $\gr R$.  If $M$ is another indecomposable graded $R$-module, then $qM$ is indecomposable by \cite{CF} and thus $qM \cong qM_i$ for some $i$.  Now Corollary~\ref{coro:isomorphicGradable} implies that $M \cong M_i[j]$ for some $j \in \mathbb{Z}$. Thus $\gl M = \gl M_i \leq \max \{\gl M_i\}_{1\leq i\leq n}$.
\end{proof}

As another consequence of Camillo's and Fuller's result, we obtain a partial converse of Corollary~\ref{coro:counitSplits}(2).

\begin{lemma}\label{lemma:counitSplits}
Assume $R$ is a right artinian graded ring, and let $N_R$ be a finitely generated $R$-module.  If $\delta_N : qpN \rightarrow N$ splits, then $N$ is gradable.
\end{lemma}  

\begin{proof}  Let $i : N \rightarrow qpN$ be a splitting for $\delta_N$. Since $N$ is f.g., the image of $i$ is contained in a finitely generated gradable submodule $qM$ of $qpN$, where $M \subseteq pN$ is a graded submodule.  Clearly $i$ still splits $\delta_N |_{qM}$, and hence $N$ is isomorphic to a direct summand of the f.g. gradable module $qM$.   Thus $N$ is gradable.
\end{proof}

 We conclude this section with our key characterization of gradable $R$-modules.  Recall that a module is {\it pure-projective} if it is isomorphic to a direct summand of a direct sum of finitely presented modules.  If $R$ is right artinian, then the Krull-Schmidt theorem implies that this is equivalent to the module being a direct sum of fintely presented (or finite length) submodules \cite{War}. 
We will review other characterizations later on.

\begin{therm}\label{thm:GradablePP} Let $R$ be a right artinian graded ring. A f.g. right $R$-module $N$ is gradable if and only if $pN$ is pure-projective.
\end{therm}

\begin{proof}
By Lemma~\ref{lemma:pullupGraded}, it suffices to show the if statement.  Clearly, we may assume that $N$ is indecomposable.  If $pN$ is pure projective, there is a finitely generated, graded direct summand $M$ of $pN$ such that $\delta_N |_{qM} : qM \rightarrow N$ is onto.  Since $M$ is a direct summand of $pN$ and $M$ has finite graded length, by Proposition~\ref{prop:NgeneratesM} we know that $qM \in \add(N)$.  Then $N$ must be a direct summand of the gradable module $qM$, and hence $N$ is gradable.

\end{proof}

\begin{remark} \emph{We wonder to what extent this result might also be true over more general graded rings.  Notice that since the class of pure projective modules is closed under direct summands, a necessary condition for $pN$ pure projective to imply that $N$ is gradable is that direct summands of gradable modules are gradable.}\end{remark}

\section{A question of Gordon and Green}  

In this section we investigate the consequences of the assumption that there is a bound on the graded lengths of the indecomposable graded $R$-modules for a graded right artinian ring $R$.  Following Gordon and Green, we set $$G_R = \sup\ \{\gl M\ |\ M_R \in \gr R\ \mbox{indecomposable}\}.$$  In the introduction of \cite{GoGr2}, Gordon and Green speculate that every f.g. $R$-module should be gradable if $G_R$ is finite.  Indeed, this result will follow from the next proposition along with our above characterization of gradable modules.

\begin{propos}\label{prop:PureProj}  Assume that $G = G_R < \infty$.  Then any locally finite graded right $R$-module $X$ is pure-projective.
\end{propos}

We postpone the proof of this proposition to the end of this section, as we will need to review some other facts about pure-projective modules first.  For now, we note some immediate consequences of this result.

\begin{proof}[Proof of Theorem~\ref{thm:FiniteG}]  If $G_R < \infty$, then $pN$ is pure-projective for any f.g. $R$-module $N$.  By Theorem~\ref{thm:GradablePP}, $N$ is gradable.  By Lemma~\ref{lemma:F2F3}, if $R$ has only finitely many f.g. indecomposable gradable modules, up to isomorphism, then $G_R < \infty$ and thus every f.g. $R$-module is gradable.  Thus $R$ has only finitely many f.g. indecomposable modules up to isomorphism.  The converse is trivial.
\end{proof}

We now prepare for the proof of Proposition~\ref{prop:PureProj}.  We begin with a simple observation.

\begin{lemma}\label{lemma:Syz} Let $d = \max \{|i|\ |\ R_i \neq 0\}$.  If a graded module $M_R$ is concentrated in degrees $[a,b]$, then its first syzygy $\Omega M$ is concentrated in degrees $[a-d,b+d]$.
\end{lemma}

\begin{proof}  We have an epimorphism from a  free $R$-module that is generated in degrees between $a$ and $b$ to $M$.  Since $R_i = 0$ for all $i$ with $|i|>d$, such a free module is concentrated in degrees $[a-d,b+d]$.  Thus the kernel of this epimorphism (of which $\Omega M$ is a direct summand) is also concentrated in degrees $[a-b,b+d]$.
\end{proof}

We now review some facts about pure-projective modules over an arbitrary ring $S$.  Let $M$ be a right $S$-module.  By a {\it tuple} $\bar{m}$ in $M$, we mean a finite tuple $(m_1,\ldots,m_n)$ of elements of $M$, which we also treat as a row-matrix, or identify with a map $S^n \rightarrow M$.  Given a pair of rectangular matrices $A$ and $B$ over $S$ with the same number of columns, we can consider the formula 
$$\varphi(\bar{x}): \exists \bar{y} (\bar{x} A = \bar{y} B)$$ in free variables $\bar{x} = (x_1,\ldots, x_n)$ where $n$ is the number of rows of $A$.  
Such a $\varphi$ is called a {\it positive primitive formula} (or a {\it pp-formula} for short).  It is satisfied by a tuple $\bar{m} = (m_1, \ldots, m_n)$ in $M$ if there exist $\bar{b} = (b_1, \ldots, b_k)$ in $M$ such that $\bar{m}A = \bar{b} B$, and then we write $M \models \varphi(\bar{m})$.  For any tuple $\bar{m}$ in $M$, we write $\pp^M(\bar{m})$ for the set of all pp-formulae satisfied by $\bar{m}$ in $M$, and we refer to this set as the {\it pp-type} of $\bar{m}$ in $M$.  It is straightforward to see that pp-types are preserved by homomorphisms: that is, if $f : M \rightarrow N$ is an $S$-module map then $\pp^M(\bar{m}) \subseteq \pp^N(f(\bar{m}))$.   We note that a monomorphism $f: M \rightarrow N$ is {\it pure} if and only if $\pp^M(\bar{m}) = \pp^N(f(\bar{m}))$ for all $\bar{m}$ in $M$ (see \cite{P}, \S 2.3).

Conversely, when $M$ is finitely presented an inclusion of pp-types $\pp^M(\bar{m}) \subseteq \pp^N(\bar{n})$ implies the existence of a morphism $f :M \rightarrow N$ such that $f(\bar{m}) = \bar{n}$  (see Fact 2.1 in \cite{PR}, or Ch. 8 of \cite{P}).

For two pp-formulas $\varphi(\bar{x})$ and $\psi(\bar{x})$ in the same number of variables, we write $\varphi \rightarrow \psi$ if for all tuples $\bar{m}$ in all $S$-modules $M$, $\psi(\bar{m})$ holds whenever $\varphi(\bar{m})$ holds.  We say that a pp-type $p = \pp^M(\bar{m})$ is {\it finitely generated} if there exists a single pp-formula $\varphi(\bar{x}) \in p$ such that $\varphi \rightarrow \psi$ for all $\psi \in p$.  In fact, every pp-type $\pp^M(\bar{m})$ in a finitely presented module $M$ is finitely generated (see Fact 2.3 in \cite{PR} or Prop. 8.4 in \cite{P}).

\begin{propos}[\cite{PR}]  Let $P$ be a countably generated $S$-module.  Then $P$ is pure-projective if and only if the pp-type of any tuple $\bar{m}$ in $P$ is finitely generated. 
\end{propos}

\begin{coro}\label{cor:PPcriterion} Let $P$ be a countably generated $S$-module such that every finite tuple $\bar{m}$ in $P$ is contained in a f.p. pure submodule of $P$.  Then $P$ is pure-projective.
\end{coro}

\begin{proof}  Suppose $\bar{m}$ is contained in the f.p. pure submodule $A$ of $P$.  Then $\pp^P(\bar{m}) = \pp^A(\bar{m})$ by the definition of pure submodule, and the latter is finitely generated since $A$ is f.p.  Thus $P$ must be pure-projective. \end{proof}

\noindent
\begin{proof}[Proof of Proposition~\ref{prop:PureProj}] Let $X = \oplus_{i \in \mathbb{Z}} X_i$ with each $X_i$ finite length over $R_0$.  For each $j \geq 0$, we set $Y_j = (\oplus_{i=-j}^j X_i)R$, which is a finitely generated graded submodule of $X$.  We have a directed system of degree-zero monomorphisms $Y_1 \stackrel{f_1}{\longrightarrow} Y_2 \stackrel{f_2}{\longrightarrow} Y_3 \stackrel{f_3}{\longrightarrow} \cdots$ such that each $f_j$ is an isomorphism in degrees $[-j,j]$.  Clearly we have $X = \varinjlim Y_j$.  For $j>i$ we will write $f_{i,j}$ for the composite $f_{j-1}\cdots f_i : Y_i \to Y_j$.

Let $\bar{p}$ be a finite tuple in $X$, and choose $j$ so that $\bar{p}$ is contained in $\oplus_{i=-j}^j X_i$, and hence also in $Y_j$.  Consider $f_{j, j+G+d}(\bar{p})$ in $Y_{j+G+d}$, where $d = \max \{|i|\ | \ R_{i} \neq 0\}$.  We can decompose $Y_{j+G+d}$ into a direct sum of indecomposable graded modules, and for such a decomposition let $A$ be a minimal direct sum of these indecomposable summands for which $f_{j,j+G+d}(\bar{p}) \subset A$, and let $B$ be a complement of $A$.  By minimality of $A$, each indecomposable summand of $A$ must be nonzero in some degree from $[-j,j]$, and thus $A$ is concentrated in degrees $[-j-G,j+G]$.  For any $i\geq 1$ consider the following commutative exact diagram in $\gr R$, where we have put $k := j+G+d$.
$$\xymatrix{ & & 0 \ar[d] & 0 \ar[d] \\ 0 \ar[r] & A \ar@{=}[d] \ar[r] & Y_{k} \ar[d]^{f_{k,k+i}} \ar[r] & B \ar[d]^g \ar[r] & 0 \\
0 \ar[r] & A \ar[r] & Y_{k+i} \ar[d] \ar[r] & B_i \ar[r] \ar[d] & 0 \\ & & C_i \ar[d] \ar@{=}[r] & C_i \ar[d] \\ & & 0 & 0}$$

The top row is the split exact sequence corresponding to the decomposition $Y_k = A \oplus B$ above, and it coincides with the pull-back of the middle row along the map $g: B \rightarrow B_i$.  We also know that $C_i = \coker g \cong Y_{k+i}/Y_{k}$ is concentrated in degrees outside of $[-k,k]$.  Hence, by Lemma~\ref{lemma:Syz} $\Omega C_i$ is concentrated in degrees outside of $[-j-G,j+G]$, and thus $\Hom_{\Gr R}(\Omega C_i, A) = 0$, meaning that $\Ext^1_{\Gr R}(C_i, A) = 0$.  Applying $\Ext^1_{\Gr R}(-,A)$ to the right-most column of the above diagram now shows that $\Ext^1_{\Gr R}(g,A)$ is a monomorphism taking the element of $\Ext^1_{\Gr R}(B_i,A)$ corresponding to the middle row to $0$.  Thus the middle row also splits.  Since the inclusion of $A$ into $Y_{k+i}$ splits for all $i\geq 0$, we see that $A$ is a pure submodule of $X =\varinjlim_i Y_{k+i}$.  That $X$ is pure-projective now follows from Corollary~\ref{cor:PPcriterion}.  
 \end{proof}

\section{Left-right symmetry}

As it is well-known that condition (F1), and equivalently (F2), that $R$ has finite representation type is left-right symmetric, we show here that the remaining conditions (F3) and (F4) are also left-right symmetric when $R$ is left and right artinian.  Our motivation here stems from the still unresolved pure-semisimplicity conjecture.  Recall that a ring $R$ is {\it right pure-semisimple} if every right $R$-module is pure-projective.   The pure-semisimplicity conjecture asserts that this is a left-right symmetric notion, and since it is known that the class of rings of finite representation type coincides with the class of left and right pure-semisimple rings (see \cite{P} for example), this conjecture is equivalent to the statement that all right pure-semisimple rings have finite representation type.  We can thus regard right pure-semisimplicity as a (potential) weakening of finite representation type.  In particular, in Theorem~\ref{thm:FiniteG} we have seen that over a graded right artinian ring $R$ of finite representation type every f.g. module is gradable, but in light of Theorem~\ref{thm:GradablePP} our proof only requires this weaker hypothesis.

\begin{coro}\label{coro:rightPSS} If $R$ is a graded right pure semisimple ring, then every finitely generated right $R$-module is gradable.
\end{coro}

In particular, this result may invite one to look for a counterexample to the pure semisimplicity conjecture in a graded ring $R$ which is left pure-semisimple, yet has an ungradable finitely generated right module.   However, no such example can exist since the condition (F4) that every f.g. right $R$-module is gradable turns out to be left-right symmetric.  This can be seen by using the Auslander-Bridger transpose $\Tr$ as follows.  Recall that if $M_R$ is finitely presented with a projective presentation $P_1 \stackrel{f}{\to} P_0 \to M \to 0$, then $\Tr M$ is defined as the cokernel of the map $f^* : P_0^* \to P_1^*$ in $R \lmod$, where $(-)^* = \Hom_R(-,R)$.  Then $\Tr$ induces a bijection between the isomorphism classes of indecomposable nonprojective left and right finitely presented $R$-modules.  Furthermore, if $R$ is graded then the argument in \cite{GoGr2}, shows that $\Tr M$ is gradable if and only if $M$ is gradable for any finitely presented right $R$-module $M$.

\begin{propos}\label{prop:gradableTr}  Let $R$ be a graded right artinian ring with $d= \max\{|n| \mid R_n \neq 0\}$.  Then a finitely presented right $R$-module $M$ is gradable if and only if $\Tr M$ is gradable.  Moreover, for any finitely generated graded module $N$, we have the inequality $$\gl (\Tr N) \leq \gl N + 4d.$$

\end{propos}

\begin{proof}  Suppose $M=qN$ is an indecomposable gradable module, with $N$ graded.  Then we may compute $\Tr N$ using a graded minimal projective presentation $P_1 \stackrel{f}{\to} P_0 \to N \to 0$ in $\gr R$.  Then $f^*: P_0^* \to P_1^*$ is a map in $R \lgr$, showing that $\Tr N = \coker f^*$ is also graded.  We have a natural isomorphism $q\Tr \cong \Tr q$, where $q$ is the forgetful functor from graded $R$-modules to all $R$-modules (for both left and right modules).  Thus it follows that $\Tr M = \Tr qN \cong q\Tr N$ is again gradable.

For the second statement, let $n = \gl N$ and assume $N$ is concentrated in degrees $[1,n]$.  Then $P_0$ is generated in degrees $[1,n]$ and thus $P_0$ is concentrated in degrees $[1-d,n+d]$.  Then $P_1$ is generated in degrees $[1-d,n+d]$ and thus $P_1^* = \Hom_R(P_1,R) = \oplus_i \Hom_{\gr R}(P_1,R[i])$ is concentrated in degrees $[-(n+2d), 2d-1]$.  Since $\Tr N$ is a quotient of $P_1^*$, it is concentrated in these same degrees and we have $\gl \Tr N \leq 2d-1+n+2d+1 = n + 4d$.
\end{proof}

\begin{coro}\label{coro:LRsymmetry}  If $R$ is a graded artinian ring, then 
\begin{enumerate}
\item Every finitely generated right $R$-module is gradable if and only if every finitely generated left $R$-module is gradable;
\item There is a bound on the graded lengths of the indecomposable graded right $R$-modules if and only if there is a bound on the graded lengths of the indecomposable graded left $R$-modules.
\end{enumerate} 
\end{coro}

Thus we obtain a strengthened version of Corollary~\ref{coro:rightPSS}.

\begin{coro}\label{coro:leftPSS} If $R$ is a graded ring that is right or left pure semisimple, then every finitely presented right and left $R$-module is gradable.  
\end{coro}

Notice that our conclusion is only for {\it finitely presented} modules.  If $R$ is left pure semisimple, then $R$ is left artinian and all finitely generated left modules are finitely presented.  However, it is not known if $R$ being left pure semisimple implies that $R$ is right artinian, and hence $\Tr$ only gives us information about the finitely presented right $R$-modules in general.

\section{Pr\"ufer modules}

Motivated by Ringel's work connecting Pr\"ufer modules to the second Brauer Thrall conjecture, we now consider Pr\"ufer modules that are related to the pull-ups of modules over graded artinian rings.  We begin with some definitions over an arbitrary ring $S$.

\begin{defin} An $S$-module $P$ is a {\bf Pr\"ufer module} if there exists a locally nilpotent, surjective endomorphism $\phi$ of $P$ such that $\ker \phi$ is nonzero and of finite length.  We call $Y := \ker \phi$ the {\bf basis} of $P$, and we write $Y[t]$ for $\ker \phi^t$.  
\end{defin}

If $(P,\phi)$ is a Pr\"ufer module with basis $Y$, we have an increasing chain of submodules $Y=Y[1] \subset Y[2] \subset Y[3] \subset \cdots$ with $P = \cup_{n \geq 1} Y[n]$.  Moreover we have short exact sequences $$\ses{Y}{Y[n+1]}{Y[n]}{}{\phi} \ \ \ and \ \ \ \ses{Y[n]}{Y[n+1]}{Y}{}{\phi^n}$$ for all $n \geq 1$.  In particular, it follows that each $Y[n]$ has finite length.  

If $X$ is a right $S$-module, we write $\Add(X)$ for the full subcategory of $\rMod S$ consisting of direct summands of arbitrary direct sums of copies of $X$.  If $X$ is a direct sum of countably generated strongly indecomposable modules $X_i$ (i.e., modules with local endomorphism rings), then Warfield's generalization of the Krull-Schmidt theorem \cite{War} implies that $\Add(X)$ consists of the modules that are direct sums of copies of the $X_i$.  In particular, this holds if $X$ has finite length.

\begin{defin}  An $S$-module $M$ has {\bf finite type} if it belongs to $\Add(X)$ for some finite length $S$-module $X$. 
\end{defin}

We now summarize the background concerning the second Brauer-Thrall conjecture, following \cite{Ringel2}.  Let $\Lambda$ be a finite-dimensional algebra over a field $k$.  The second Brauer-Thrall conjecture, proved by Bautista when $k$ is algebraically closed, asserts that if $\Lambda$ has infinite representation type then there are infinitely many natural numbers $d$ such that there are infinitely many non-isomorphic indecomposable $\Lambda$-modules of length $d$.  For an artin algebra $\Lambda$, we can modify this conjecture by considering instead the endo-lengths of $\Lambda$-modules, that is, their lengths as modules over their endomorphism rings.  Then the conjecture becomes: if $\Lambda$ has infinite representation type then there are infinitley many natural numbers $d$ such that there are infinitely many non-isomorphic indecomposable $\Lambda$-modules of endo-length $d$.  
In fact, if $\Lambda$ is a finite-dimensional algebra over any field $k$, Crawley-Boevey has shown that this conclusion will hold provided $\Lambda$ has a {\it generic} module \cite{CB}, where we recall a $\Lambda$-module $M$ is generic if it is indecomposable of finite endo-length but not finitely generated.  Ringel has shown that Pr\"ufer modules yield generic modules.

\begin{therm}[3.4 in \cite{Ringel2}]   The following are equivalent for a Pr\"ufer module $M$ over an artin algebra $\Lambda$. 
\begin{enumerate}
\item $M$ is not of finite type.
\item There is an infinite index set $I$ such that the product module $M^I$ has a generic direct summand.
\item For every infinite index set $I$, the product module $M^I$ has a generic direct summand.
\end{enumerate}
\end{therm}

To show that a Pr\"ufer module $M$ is not of finite type, it suffices to show that it is not pure-projective.  In fact, it turns out that these two notions are equivalent for Pr\"ufer modules.  Although we don't need this fact here, we include a short proof that may be of independent interest.  We make use of the ``telescoping map theoerm'' from \cite{PR}. 

\begin{therm}[\cite{PR}]  Suppose that $M$ is a countably generated pure-projective module.  If $M$ is the direct limit of a direct system $(N_i, f_{ij})$, $i, j \in I$, then $M \oplus \bigoplus_{i \in I} N_i \cong \bigoplus_{i \in I} N_i$.
\end{therm}

In the above theorem, suppose that all of the $N_i$ have finite length.  Then $\bigoplus_{i \in I} N_i$ is a direct sum of strongly indecomposable modules, and hence $M$ too must be a direct sum of strongly indecomposable modules, which are in addition finite length direct summands of the $N_i$.  Since Pr\"ufer modules are defined in terms of finite length modules, and pure projectives in terms of finitely presented modules, we will need to add the assumption on $S$ that all finite length $S$-modules are finitely presented.  Of course, this is equivalent to assuming that all maximal right ideals of $S$ are finitely generated.  Under this assumption, every finite-type $S$-module is pure-projective.

\begin{propos}\label{prop:PPprufer} Assume all simple $S$-modules are f.p., and let $P_S$ be a Pr\"ufer module.  Then $P$ is pure-projective if and only if it has finite type.
\end{propos}

\begin{proof}  Assume that $(P,\phi)$ is a pure-projective Pr\"ufer module with basis $Y$.  By the telescoping map theorem and the remarks following it,   $P$ belongs to $\Add(\bigoplus_{i \geq 1} Y[i])$ and we can write $P = \oplus_{j \in J} Q_j$ for indecomposable finite length modules $Q_j$.  Since $Y \subseteq P$ is finitely generated, it must be contained in a direct summand of $P$ of the form $A:=\oplus_{j \in J_0} Q_j$, where $J_0$ is a finite subset of $J$.  Write $Q = \oplus_{j \in J\setminus J_0} Q_j$ so that $P = A \oplus Q$.  Since $A$ is finitely generated, we know $A \subseteq Y[n]$ for $n$ sufficiently large.  In fact, since $A$ is a direct summand of $P$, it is also a direct summand of each such $Y[n]$, and we can even write $Y[n] = A \oplus B_n$, where we set $B_n := Y[n] \cap Q$, for all $n$ sufficiently large, say for $n \geq N$.  Since $Y = \ker \phi \subseteq A$, the exact sequence $\ses{Y}{Y[n+1]}{Y[n]}{}{\phi}$ shows that $Y[n] \cong A/Y \oplus B_{n+1}$.  Since we also have $Y[n] = A \oplus B_n$, it follows that $Y[n+1] = A \oplus B_{n+1} \in \add(Y[n])$, for all $n \geq N$.  Thus $\Add(\bigoplus_{i \geq 1} Y[i]) = \Add(\bigoplus_{i=1}^{N} Y[i])$, and hence $P$ has finite type.
\end{proof}

We now return to a graded right artinian ring $R$.  As in Section 3, we assume that $R$ is basic so that its Jacobson radical $J$ satisfies $J = \rad R_0 \oplus \bigoplus_{n\neq 0} R_n$.  Thus any f.g. $R$-module $M$ is f.g. over $R_0$ and the length of $M$ is the same over $R$ as over $R_0$.  We write $R_{\geq d}, R_{>d}, R_{\leq d}, \ldots$ for $\oplus_{n \geq d} R_n$ and so on.

  For any f.g. right $R$-module $M$ we can define a Pr\"ufer module $P_M$ as the quotient of $pM$ by the submodule generated by $\oplus_{i<0} Me_i$.  The corresponding endomorphism $\phi$ of $P_M$ is induced by the automorphism $\psi$ of $pM$ that sends $me_i$ to $me_{i-1}$ for all $m \in M$ and all $i \in \mathbb{Z}$:

$$\xymatrix{0 \ar[r] & (\oplus_{i<0} Me_i)R \ar[r] \ar[d]^{\psi_0} & pM \ar[r] \ar[d]^{\psi}_{\cong} & P_M \ar@{-->}[d]^{\phi} \ar[r] & 0 \\ 0 \ar[r] & (\oplus_{i<0} Me_i)R \ar[r]  & pM \ar[r]  & P_M \ar[r] & 0} $$

By the snake lemma, $\phi$ is onto and $\ker \phi \cong \coker\ \psi_0$.  Clearly $\psi_0$ maps onto the degree $i$ part of $(\oplus_{i<0} Me_i)R$ for all $i< -1$.  For $d \geq 0$, the cokernel of $\psi_0$ in degree $d$ is $MR_{\geq d+1}/MR_{\geq d+2}$ (as an $R_0$-module).  While in degree $-1$, the cokernel of $\psi_0$ is isomorphic to $M/MR_{\geq 1}$, which is nonzero since $R_{\geq 1} \subseteq J$ and $M/MJ \neq 0$ by Nakayama's lemma. 
Thus the kernel of $\phi$ can be described as $$\ker \phi \cong \coker\ \psi_0 \cong M/MR_{\geq 1} \oplus \bigoplus_{d \geq 0} (MR_{\geq d+1}/MR_{\geq d+2}),$$  which can be thought of as the `positive' associated graded module of $M$.  If $R$ is positively graded, this coincides with usual associated graded module of $M$, up to a degree shift.  The $R$-module action is induced by the $R$-action on $M$. 
Furthermore, this kernel has finite length over $R_0$, and hence over $R$, since $M$ does.

As it is also clear that $\phi$ is locally nilpotent, since $\psi^{n+1}(Me_n) \subseteq (\oplus_{i<0}Me_i)R$ for every $n \geq 0$, we conclude that $P_M$ is a Pr\"ufer module.   We remark that (at least in some cases) it appears $P_M$ can also be obtained using Ringel's ladder construction of Pr\"ufer modules \cite{Ringel1}.  

\begin{propos}\label{prop:Prufer} With notation as above, the Pr\"ufer module $P_N$ has finite type if and only if $N_R$ is gradable.  
\end{propos}

\begin{proof} First assume $N_R$ is gradable and write $N = qM$ for a graded module $M$.  We may assume that $M$ is concentrated in degrees $0$ and above. Then $pN =pqM \cong \oplus_{i \in \mathbb{Z}} M[i]$.  If $d = \max\{|i| \mid R_i \neq 0\}$, then the submodule $(\oplus_{i<0} Ne_i)R$ of $pN$ is concentrated in degrees less than $d$, and thus $P_N$, which is defined as the quotient of $pN$ by this submodule, will be the direct sum of $\oplus_{i \geq d} M[-i]$ and a finite length submodule generated in degrees $[0,d-1]$. 

Conversely, assume that $P_N$ has finite type.   We can write $P_N = M' \oplus M''$ where $M'$ is a f.g. graded direct summand of $P_N$ that contains $(P_N)_i$ for all $0\leq i \leq 2d$.  Now choose $M$ to be a f.g. graded direct summand of $M''$ that contains $(P_N)_k$ for some $k$.  In particular $M$ is concentrated in degrees larger than $2d$.  We claim that $M$ is also isomorphic to a direct summand of $pN$.  To see this, let $U$ be the graded submodule of $P_N$ generated by all elements in degrees $> 2d$.  The map $U \to P_N \to M$ is still onto, and the splitting of the map $P_N \to M$ factors through the inclusion $U \to P_N$.  Since $U$ is generated in degrees $>2d$, $U$ is concentrated in degrees $>d$, and thus the inclusion $U \to P_N$ factors through the projection $pN \to P_N$, which is an isomorphism in degrees $>d$.  It follows that the map $M \to U \to pN$ splits the $pN \to P_N \to M$, and hence $M$ is a direct summand of $pN$.  Moreover, since $M$ was chosen to contain $(P_N)_k$, the composition of the inclusion $qM \to qpN$ with the natural map $\delta_N : qpN \to N$ is onto.  Now, as in the proof of Theorem~\ref{thm:GradablePP}, we see that $N$ is isomorphic to a direct summand of $M$ and hence is gradable.

\end{proof}

\begin{coro}\label{cor:BT2} Let $\Lambda$ be a finite-dimensional algebra over a field $k$.  If $\Lambda$ admits a grading in which a f.g. $\Lambda$-module $M$ is not gradable, then $\Lambda$ has a Pr\"ufer module that is not of finite type.  Hence $\Lambda$ has a generic module.
\end{coro}

\begin{remark} \emph{There do exist algebras for which every grading is trivial in the sense of \cite{GoGr1}.   In particular, there are algebras $\Lambda$ of infiinite representation type for which the hypotheses of the above corollary can never be satisfied: i.e., for any possible grading of $\Lambda$, every $\Lambda$-module is gradable.  For example, take $\Lambda = kQ$ where $Q$ is a tree that is not Dynkin.} 
\end{remark}
\section{Some open questions}

In this section we propose several interesting questions about graded artinian rings and their modules.  These questions are mostly motivated by our attempts to generalize other results of Gordon and Green for graded artin algebras.  For a graded artin algebra $\Lambda$, Gordon and Green have shown that any component of the AR-quiver of $\Lambda$ that contains a gradable module must consist entirely of gradable modules \cite{GoGr2}.  In particular, since any indecomposable projective $\Lambda$-module is gradable, every inececomposable module in a component of the AR-quiver of $\Lambda$ that contains a projective module must also be gradable.  Now, Auslander and Smal\o \ have shown that $\ind \Lambda$ always has a preprojective partition, and each indecomposable preprojective module is a successor of an indecomposable projective module in the AR-quiver of $\Lambda$ \cite{AS}.  Thus every preprojective $\Lambda$-module is gradable.  In our setting we can use the following definition of preprojective modules from \cite{TFRT}, based on one of the equivalent conditions in Theorem 5.1 of \cite{AS}.

\begin{defin} An indecomposable $R$-module $Y_R$ is {\bf preprojective} if there exists a finitely presented module $X_R$ such that $Y$ has no direct summands in $\add(X)$, and for every non-split epimorphism $f : Z \to Y$ in $\rmod R$, $Z$ must have a direct summand in $\add(X)$.
\end{defin}

Thus, hoping to generalize the situation for graded artin algebras, we propose the following.

\begin{question}\label{Q:preprojective} For a graded right artinian ring $R$, is every preprojective right $R$-module gradable?
\end{question}

In general, a right artinian ring $R$ does not have an AR-quiver, but there is another nice consequence of Gordon's and Green's AR-quiver result that still makes sense in our context, and which we believe should hold.  Recall that if $\C$ is an additive subcategory of an additive category $\A$, a morphism $f: C \rightarrow X$ is a {\it right $\C$-approximation} of $X$ if $C \in \C$ and every map $g: C' \rightarrow X$  with $C' \in \C$ factors through $f$.  The subcategory $\C$ is {\it contravariantly finite} in $\A$ if every $X \in \A$ has a right $\C$-approximation. In \cite{CaHa}, Carlson and Happel show that the indecomposable objects of a proper contravariantly finite subcategory $\C$ of $\rmod \Lambda$ cannot consist of a union of connected components of the AR-quiver of $\Lambda$.  In fact, the proof of Theorem 2.1 in \cite{CaHa} establishes the following:
 
 \begin{propos}[Theorem 2.1 in \cite{CaHa}]  If $f: X \rightarrow M$ is a right $\C$-approximation of a $\Lambda$-module $M$ that does not belong to $\C$.  Then $X$ contains an indecomposable direct summand $U$ for which there is an irreducible map $g : U \rightarrow V$ with $V$ not in $\C$.
  \end{propos} 
 
 \begin{coro} Let $(\Gamma_i)_{i \in I}$ be a collection of components of the AR-quiver of $\Lambda$.  Suppose that $\C$ is the full subcategory of $\rmod \Lambda$ consisting of all direct sums of modules from the various $\Gamma_i$.  Then $M$ has a right $\C$-approximation if and only if $M \in \C$.
 \end{coro}
 
Since the subcategory $\G^f_\Lambda$ of f.g. gradable $\Lambda$-modules satisfies the hypothesis of the corollary, we know that a f.g. $\Lambda$-module $M$ has a right $\G^f_\Lambda$-approximation if and only if it is gradable.

\begin{question}\label{Q:gradedApproximation} Let $R$ be a graded right artinian ring and suppose that the finitely generated module $N_R$ has a right $\G^f_R$-approximation.  Does it follow that $N$ is gradable?
\end{question}

In a different direction, we do not know whether $(F4) \Rightarrow (F3)$ holds in general.  In fact, the only examples we know where every f.g. right $R$-module is gradable occur when $R$ is either of finite representation type or else graded equivalent (in the sense of \cite{GoGr1}) to a trivially graded ring.  In either case, it follows that $G_R < \infty$.

\begin{question}\label{Q:F4F3} Suppose $R$ is a right artinian graded ring such that every f.g. right $R$-module is gradable.  Does it follow that $G_R$ is finite?
\end{question}

Finally, we point out that the implication $(F4) \Rightarrow (F1)$ fails trivially.  To see this, take any right artinian ring $R$ of infinite representation type and give it a trivial grading (e.g., $R=R_0$).  Then every $R$-module $M$ is trivially gradable (e.g., $M=M_0$).

\end{document}